\documentclass[preprint,12pt]{elsarticle}
\usepackage{array}
\usepackage{bigstrut}
\usepackage{graphicx}
\usepackage{epsfig}
\usepackage{amssymb}
\usepackage{rotating}
\usepackage{rotate}
\usepackage{longtable}
\usepackage{amsmath}
\usepackage{color}

\newcommand{\q}{\quad}
\newcommand{\ee}{{\rm e}\hspace{1pt}}
\newcommand{\dd}{\hspace{0.5pt}{\rm d}\hspace{0.5pt}}

\usepackage{amsthm}
\usepackage{lineno}
\usepackage{amssymb,latexsym,amscd,amsmath,amsfonts,enumerate,supertabular}
\usepackage{tabularx}

\journal{J.\ of Computational and Applied Mathematics}

\numberwithin{equation}{section}
\newtheorem{lemma}[]{Lemma}
\numberwithin{lemma}{section}
\newtheorem{theorem}[]{Theorem}
\numberwithin{theorem}{section}

\numberwithin{example}{section}
\definecolor{orange}{rgb}{1,0,0}

\begin{document}

\begin{frontmatter}
\title{{\bf Explicit exponential Runge--Kutta methods \\ of high order for parabolic problems \tnoteref{label1}}}
\tnotetext[label1]{This work was supported by the FWF doctoral program `Computational Interdisciplinary Modelling' W1227. The work of the first author was in addition supported by the Tiroler Wissenschaftsfond
grant UNI-0404/1284.}
\author{Vu Thai Luan and Alexander Ostermann}
\ead{vu.thai-luan@uibk.ac.at, alexander.ostermann@uibk.ac.at}
\address{Institut f\"ur Mathematik, Universit\"at Innsbruck, Technikerstr.~19a, \\ A--6020 Innsbruck, Austria}

\begin{abstract}
\small
\indent Exponential Runge--Kutta methods constitute efficient integrators for semilinear stiff problems. So far, however, explicit exponential Runge--Kutta methods are available in the literature up to order~4 only. The aim of this paper is to construct a fifth-order method. For this purpose, we make use of a novel approach to derive the stiff order conditions for high-order exponential methods. This allows us to obtain the conditions for a method of order~5 in an elegant way. After stating the conditions, we first show that there does not exist an explicit exponential Runge--Kutta method of order~5 with less than or equal to~6 stages. Then, we construct a fifth-order method with 8 stages and prove its convergence for semilinear parabolic problems. Finally, a numerical example is given that illustrates our convergence bound.
\small
\end{abstract}

\end{frontmatter}
%% SECTION 1: ------------------------------------------------------------------
\section{Introduction}\label{sc1}
In this paper, we are concerned with the construction of high-order
exponential Runge--Kutta methods for the time discretization of stiff semilinear problems
\begin{equation} \label{eq1.1}
u'(t)=F(u(t)) = A u(t) + g(u(t)), \qquad u(t_0)=u_0
\end{equation}
on the interval $t_0\le t\le T$.
Parabolic partial differential equations, written as abstract ordinary differential equations in some Banach space and their spatial discretizations are typical examples of such problems.
Our main interest is the case of stiff problems where $A$ has a large norm or is an
unbounded operator. On the other hand, the nonlinearity $g$ is assumed to satisfy a local Lipschitz condition with a moderate Lipschitz constant in a strip along the exact solution.

In recent years, exponential integrators turned out to be very competitive for stiff problems, see
\cite{HLS98,KT05}. For a detailed overview of such integrators and their implementation, we refer to \cite{HO10}.
The main idea behind these methods is to treat the linear part of problem (\ref{eq1.1}) exactly and the nonlinearity
in an explicit way. An early paper by Friedli \cite{F78} derived the class of exponential Runge--Kutta methods for non-stiff problems. He used classical Taylor series expansion for this purpose. For stiff problems, methods of orders up to four were constructed in \cite{HO05b}. Very recently, in \cite{LO12b}, we proposed a new and simple approach to derive stiff order conditions for exponential Runge--Kutta methods up to order five. With these order conditions at hand, we are going to construct a fifth-order method in this paper. We first show that there does not exist an explicit exponential Runge--Kutta method of order~5 with less than or equal to 6~stages. On the other hand, we are able to construct a family of fifth-order methods with 8~stages. These methods satisfy some of the order conditions only in a weakened form, however, this will be sufficient for obtaining convergence results for semilinear parabolic problems. We note that exponential Runge--Kutta methods may also be applied to the solution of other problems, see for example \cite{DP11,D09}.

The paper is organized as follows. In section \ref{sc2}, we recall
exponential Runge--Kutta methods for our further analysis.
Our abstract framework is given in section \ref{sc3}. In section \ref{sc4}, we recall the stiff order conditions derived in \cite{LO12b}. A new convergence result is also given in this section. Section~\ref{sc5} is devoted to the construction of exponential Runge--Kutta methods of order~5. In section 6, a numerical experiment is presented to illustrate the order of the new method. The main results of the paper are Theorem~\ref{th4.1}, Theorem~\ref{th5.1} and the new integrator $\mathtt{expRK5s8}$.

Throughout the paper, $C$ will denote a generic constant that may have different values at
different occurrences.
%% SECTION 2: ------------------------------------------------------------------
\section{Exponential Runge--Kutta methods in a restated form}
\label{sc2}
For solving \eqref{eq1.1}, we consider the following class of $s$-stage explicit exponential Runge--Kutta methods \cite{HO05b}, reformulated as in \cite{LO12b}
\begin{subequations} \label{eq2.1}
\begin{align}
 U_{ni}&= u_n + c_i h_n \varphi _{1} ( c_i h_n A)F(u_n) +
 h_n \sum_{j=2}^{i-1}a_{ij}(h_n A) D_{nj}, \  1\leq i\leq s,  \label{eq2.1a} \\
u_{n+1}& = u_n + h_n \varphi _{1} ( h_n A)F(u_n) + h_n \sum_{i=2}^{s}b_{i}(h_n A) D_{ni}  \label{eq2.1b}
\end{align}
with
\begin{equation} \label{eq2.1c}
  D_{ni}= g ( U_{ni})- g(u_n ), \qquad  2\leq i\leq s.
\end{equation}
\end{subequations}
As usual, $h_n = t_{n+1}-t_n >0$ denotes the time step size and the $c_i $ are the nodes. The internal stages $U_{ni}$ approximate the exact solution at $t_n +c_i h_n$. For the sake of completeness, one can define $U_{n1}=u_n$ and $c_1=0$. However, we note that these quantities do not enter the scheme anyway. The coefficients $a_{ij}(z)$ and $b_i (z)$ are chosen as linear combinations of the entire functions $\varphi_{k}(z)$ and scaled versions thereof. These functions are given by
\begin{equation} \label{eq2.2}
\varphi_{0}(z)=\ee^z, \q
\varphi_{k}(z)=\int_{0}^{1} \ee^{(1-\theta )z} \frac{\theta ^{k-1}}{(k-1)!}\,\text{d}\theta , \quad k\geq 1.
\end{equation}
They satisfy the recurrence relation
\begin{equation} \label{eq2.3}
\varphi _{k+1}(z)=\frac{\varphi_{k}(z)-\varphi_{k}(0)}{z}, \q k\geq 0.
\end{equation}
The reformulated scheme \eqref{eq2.1} can be implemented more efficiently than its original form, and it
offers many advantages for the error analysis (see \cite{LO12b}).

As shown in \cite{HO05b}, exponential Runge--Kutta methods are invariant under the transformation of
non-autonomous problems
\begin{equation} \label{eq2.3a}
u'(t)=F(t,u(t))=Au(t)+g(t,u(t))
\end{equation}
to \eqref{eq1.1} by adding $t'=1$. Scheme \eqref{eq2.1} can be applied to \eqref{eq2.3a} by replacing $F(u_n)$ in \eqref{eq2.1a}, \eqref{eq2.1b} by $F(t_n,u_n)$ and $D_{ni}$ in \eqref{eq2.1c} by
\begin{equation*} \label{eq2.3b}
D_{ni}=g (t_n +c_i h_n, U_{ni})-g (t_n, u_n).
\end{equation*}
As a consequence of this invariance we will consider henceforth the autonomous case only. However, all stated results stay \emph{mutatis mutandis} valid for the non-autonomous problem \eqref{eq2.3a} as well.
%% section 3-------------------------------------
\section{Analytical framework}\label{sc3}
Our analysis will be based on an abstract framework of analytic semigroups
on a Banach space $X$ with norm $\| \cdot \|$. Background information on semigroups can be found in the
monograph \cite{PAZY83}.

Throughout the paper we consider the following assumptions.

{\em Assumption 1. The linear operator $A$ is the infinitesimal generator of an analytic semigroup
$\ee^{tA}$ on $X$.}

This assumption implies that there exist constants $M$ and $\omega $ such that
\begin{equation} \label{eq3.1}
\|\ee^{tA}\|_{X\leftarrow X}\leq M\ee^{\omega t}, \quad t\geq 0.
\end{equation}
In particular, the expressions $\varphi_k(h_nA)$ and consequently the coefficients $a_{ij}(h_n A)$
and $b_{i}(h_n A)$ of the method are bounded operators, see \eqref{eq2.2} for $z=h_n A$. This property is crucial in our proofs.

Under Assumption 1, the following stability bound was proved in \cite[Lemma 1]{HO05a}: There exists a constant $C$ such that
\begin{equation} \label{eq3.2}
\left \|h A \sum_{j=1}^{n}\ee^{j h A} \right\|_{X\leftarrow X}  \leq C.
\end{equation}
This bound holds uniformly for all $n\geq 1$ and $h>0$ with $0<nh\le T-t_0$. We will employ this bound later on.

For high-order convergence results, we require the following regularity assumption.

{\em Assumption 2. We suppose that (\ref{eq1.1}) possesses a sufficiently smooth solution $u:[t_0, T]\to X$
with derivatives in $X$ and that $g:X \to X$ is sufficiently often Fr\'echet differentiable in a strip
along the exact solution. All occurring derivatives are assumed to be uniformly bounded.}

Assumption 2 implies that $g$ is locally Lipschitz in a strip along the exact solution.
It is well known that semilinear reaction-diffusion-advection equations can be put into
this abstract framework, see \cite{H81}.
%% section 4----------------------------------------------------------
\section{Local error, stiff order conditions and convergence results for fifth-order methods}\label{sc4}
In this section, we recall some notations and results from \cite{LO12b} that will be used later. We further give a result which allows us to relax the stiff order conditions so that high-order convergence is still guaranteed.
\subsection{Local error}\label{subsc4.1}
For the error analysis of scheme (\ref{eq2.1}), we consider one step with initial value $\tilde{u}_n=u(t_n)$ on the exact solution, i.e.
\begin{subequations} \label{eq4.1}
\begin{align}
\widehat{U}_{ni}&=\tilde{u}_n + c_i h_n \varphi_{1} ( c_i h_n A)F(\tilde{u}_n) +
h_n \sum_{j=2}^{i-1}a_{ij}(h_n A)  \widehat{D}_{nj}, \label{eq4.1a}\\
\hat{u}_{n+1}&= \tilde{u}_n + h_n \varphi _{1} ( h_n A)F(\tilde{u}_n) +
h_n \sum_{i=2}^{s}b_{i}(h_n A)  \widehat{D}_{ni} \label{eq4.1b}
\end{align}
\end{subequations}
with
\begin{equation} \label{eq4.2}
\widehat{D}_{ni}= g ( \widehat{U}_{ni})- g(\tilde{u}_n ), \q \widehat{U}_{ni}\approx u(t_n +c_i h_n).
\end{equation}
Let $\tilde{u}^{(k)}_n$ denote the $k$th derivative of the exact solution  $u(t)$ of (\ref{eq1.1}),
evaluated at time $t_n$. For $k=1, 2$ we use the common notation $\tilde{u}'_n, \tilde{u}''_n$
for simplicity. We further denote the $k$th derivative of $g(u)$ with respect to $u$ by $g^{(k)} (u)$.

Let $\tilde{e}_{n+1}=\hat{u}_{n+1}- \tilde{u}_{n+1}$
denote the local error, i.e., the difference between the numerical solution $\hat{u}_{n+1}$ after one step starting
from $\tilde{u}_{n}$ and the corresponding exact solution of (\ref{eq1.1}) at $t_{n+1}$, and let
\begin{equation}\label{eq4.3}
\psi_{j} (h_nA)= \sum_{i=2}^{s}b_{i}(h_nA)\frac{c^{j-1}_{i}}{(j-1)!}- \varphi _{j} (h_nA), \q j\geq 2.
\end{equation}
For simplicity, we also use the abbreviations $a_{ij}=a_{ij}(h_nA)$, $b_{i}=b_{i}(h_nA), \\
\varphi_{j,i}=\varphi_{j} (c_i h_nA)$ and
\begin{equation}\label{eq4.4}
\psi_{j,i}=\psi_{j,i} (h_nA)= \sum_{k=2}^{i-1}a_{ik}\frac{c^{j-1}_{k}}{(j-1)!}- c^{j}_{i}\varphi_{j,i}.
\end{equation}
In \cite{LO12b}, we have shown that
\begin{equation}\label{eq4.5}
\begin{aligned}
\tilde{e}_{n+1}&=h^2_n\psi _{2} ( h_n A)\,\mathbf{L}_n + h^3_n\psi _{3} ( h_n A) \,\mathbf{M}_n
+  h^4_n \psi_{4} ( h_n A) \,\mathbf{N}_n \\
&\quad + h^5_n \psi _{5} ( h_n A) \,\mathbf{P}_n +\mathbf{R}_n+ \mathcal{O}(h^6_n)
\end{aligned}
\end{equation}
with
\begin{equation} \label{eq4.6}
\begin{aligned}
\mathbf{L}_n&=g'(\tilde{u}_n) \tilde{u}'_n,\\
\mathbf{M}_n&=g'(\tilde{u}_n)\tilde{u}''_n+g''(\tilde{u}_n)(\tilde{u}'_n,\tilde{u}'_n),    \\
\mathbf{N}_n&=g'(\tilde{u}_n)\tilde{u}^{(3)}_n +3g'' (\tilde{u}_n)(\tilde{u}'_n,\tilde{u}''_n)
+ g^{(3)} (\tilde{u}_n)(\tilde{u}'_n,\tilde{u}'_n, \tilde{u}'_n),    \\
 \mathbf{P}_n&=g'(\tilde{u}_n)\tilde{u}^{(4)}_n+ 3g'' (\tilde{u}_n)(\tilde{u}''_n, \tilde{u}''_n )
+ 4g'' (\tilde{u}_n)(\tilde{u}'_n, \tilde{u}^{(3)}_n)  \\
&\quad + 6g^{(3)} (\tilde{u}_n)(\tilde{u}'_n,\tilde{u}'_n, \tilde{u}''_n)
+ g^{(4)} (\tilde{u}_n)(\tilde{u}'_n,\tilde{u}'_n, \tilde{u}'_n, \tilde{u}'_n),
\end{aligned}
\end{equation}
and the remaining terms
\begin{align}
&\mathbf{R}_n= h^3_n \sum_{i=2}^{s}b_{i} g'(\tilde{u}_n) \psi_{2,i} \,\mathbf{L}_n
+ h^4_n \sum_{i=2}^{s}b_{i} g'(\tilde{u}_n) \psi_{3,i} \,\mathbf{M}_n \notag \\
& + h^4_n  \sum_{i=2}^{s}b_{i} g'(\tilde{u}_n) \sum_{j=2}^{i-1}a_{ij} g'(\tilde{u}_n) \psi_{2,j}
\,\mathbf{L}_n + h^4_n \sum_{i=2}^{s}b_{i} c_i g''(\tilde{u}_n)\big(\tilde{u}'_n,
\psi_{2,i} \,\mathbf{L}_n \big) \notag  \\
&+ h^5_n \sum_{i=2}^{s}b_{i} g'(\tilde{u}_n)\psi_{4,i}\,\mathbf{N}_n
+ h^5_n \sum_{i=2}^{s}b_{i} g'(\tilde{u}_n) \sum_{j=2}^{i-1}a_{ij} g'(\tilde{u}_n) \psi_{3,j}\,\mathbf{M}_n
\allowdisplaybreaks[4] \notag  \\
& + h^5_n \sum_{i=2}^{s}b_{i} g'(\tilde{u}_n) \sum_{j=2}^{i-1}a_{ij} g'(\tilde{u}_n)\sum_{k=2}^{j-1}a_{jk}
g'(\tilde{u}_n) \psi_{2,k} \,\mathbf{L}_n \label{eq4.6a} \\
&+ h^5_n \sum_{i=2}^{s}b_{i} g'(\tilde{u}_n) \sum_{j=2}^{i-1}a_{ij}c_j g''(\tilde{u}_n)\big(\tilde{u}'_n, \psi_{2,j}
\,\mathbf{L}_n\big)
+ h^5_n \sum_{i=2}^{s}b_{i} c_i g''(\tilde{u}_n)\big(\tilde{u}'_n, \psi_{3,i} \,\mathbf{M}_n \big) \notag  \\
& + h^5_n \sum_{i=2}^{s}b_{i} c_i g''(\tilde{u}_n)\big(\tilde{u}'_n,\sum_{j=2}^{i-1}a_{ij} g'(\tilde{u}_n) \psi_{2,j}
\,\mathbf{L}_n \big)
+ h^5_n \sum_{i=2}^{s} \frac{b_{i}}{2!} g''(\tilde{u}_n)\big(\psi_{2,i} \,\mathbf{L}_n ,
\psi_{2,i} \,\mathbf{L}_n \big) \notag  \\
& + h^5_n \sum_{i=2}^{s}b_{i} \frac{c^2_i}{2!} g''(\tilde{u}_n)\big(\tilde{u}''_n,\psi_{2,i} \,\mathbf{L}_n\big)
+ h^5_n \sum_{i=2}^{s}b_{i} \frac{c^2_i}{2!} g^{(3)} (\tilde{u}_n)\big(\tilde{u}'_n,
\tilde{u}'_n,\psi_{2,i} \,\mathbf{L}_n \big). \notag
\end{align}
The remainder term in the asymptotic expansion \eqref{eq4.5} has the form
\begin{equation}\label{eq:differr}
\sum_{i=2}^s b_i(hA)\mathcal{R}_{ni}-\mathcal{R}_{n}
\end{equation}
with
\begin{equation*}
\begin{aligned}
\mathcal{R}_{ni} &=h^6_{n} \int_{0}^{1} \frac{(1-\theta )^4}{4!} g^{(5)}(\tilde{u}_n + \theta  h_n V_i)(\underbrace{V_i,\ldots, V_i}_{5 \ \text{times}}) \dd\theta ,\\
\mathcal{R}_n&=h^6_{n} \int_{0}^{1} \ee^{(1-\theta )h_n A} \int_{0}^{1} \frac{\theta^{5} (1-s)^4 } {4!}  g^{(5)} (\tilde{u}_n + s\theta  h_{n} V)(\underbrace{V,\ldots, V}_{5\ \text{times}}) \,\text{d}s \,\text{d}\theta .
\end{aligned}
\end{equation*}
Note that the arguments
\begin{equation*}
\begin{aligned}
V_i &=  \frac{1}{h_n}\Big( \widehat{U}_{ni}-\tilde{u}_n \Big) = c_i \varphi _{1} ( c_i h_n A)F(\tilde{u}_n)+ \sum_{j=2}^{i-1}a_{ij}(h_n A) \widehat{D}_{nj}, \\
V&= \frac{1}{\theta h_n}\bigl(u(t_n +\theta h_n)-u(t_n)\bigr)
\end{aligned}
\end{equation*}
remain bounded as $h_n\to 0$. Therefore, \eqref{eq:differr} is bounded by $C h^6_n$, where the constant $C$ only depends on values that are uniformly bounded by Assumptions~1 and~2. This justifies the notation $\mathcal{O}(h^6_n)$ for the remainder in~\eqref{eq4.5}.
%%%-----------------------------------
\subsection{Stiff order conditions for methods of order five}\label{subsc4.2}
By zeroing the corresponding terms in \eqref{eq4.5}, the stiff order conditions for methods of order five can easily be identified, see \cite[Table~1]{LO12b}. We note, however, that the first and the third condition of that table are automatically satisfied in our context since they were used to derive our reformulated scheme \eqref{eq2.1}. As we will heavily need the order conditions for constructing a method of order five, we display them again in Table~1 below, where we have omitted the two redundant conditions.
{
\setlength{\extrarowheight}{1.6 pt}
\renewcommand{\arraystretch}{1.35}
\vspace{-3mm}
\begin{table}[ht!]
\begin{center}
\caption{Stiff order conditions for explicit exponential Runge--Kutta methods up to order 5. The variables
$Z$, $J$, $K$, $L$ denote arbitrary square matrices, and $B$ an arbitrary bilinear mapping of appropriate
dimensions. The functions $\psi_{k,l}$ are defined in \eqref{eq4.4}.}\label{tb1}
\normalsize
\vspace{1.4mm}
\begin{tabular}{ |c|c|c| }
\hline
{\ No.\ } & Stiff order condition & {\ Order\ } \\
\hline
% 1& $\sum_{i=1}^{s}b_{i}(Z) = \varphi_1(Z)$& 1 \\
% \hline
1& $\sum_{i=2}^{s}b_{i}(Z)c_i = \varphi_2(Z)$& 2 \\
% 3& $\sum_{j=2}^{s}a_{ij}(Z) = c_i\varphi_1(c_i Z), \quad i=2,\ldots,s$&2\\
 \hline
2& $\sum_{i=2}^{s}b_{i}(Z)\frac{c_i^2}{2!} = \varphi_3(Z)$&3 \\
3& $\sum_{i=2}^{s}b_{i}(Z) J \psi _{2,i}(Z)=0 $& 3 \\
 \hline
4&$\sum_{i=2}^{s}b_{i}(Z)\frac{c_i^3}{3!} = \varphi_4(Z)$&4 \\
5& $\sum_{i=2}^{s}b_{i}(Z) J \psi _{3,i}(Z)=0 $& 4 \\
6& $\sum_{i=2}^{s}b_{i}(Z) J \sum_{j=2}^{i-1}a_{ij}(Z) J \psi_{2,j}(Z)=0 $& 4 \\
7 & $\sum_{i=2}^{s}b_{i}(Z) c_i K\psi_{2,i}(Z)=0 $& 4 \\
\hline
8&$\sum_{i=2}^{s}b_{i}(Z)\frac{c_i^4}{4!} = \varphi_5(Z)$& 5 \\
9&$ \sum_{i=2}^{s}b_{i}(Z) J \psi_{4,i}(Z)=0 $& 5 \\
10&$\sum_{i=2}^{s}b_{i}(Z) J \sum_{j=2}^{i-1}a_{ij}(Z) J \psi_{3,j}(Z)=0 $& 5 \\
11&\quad $\sum_{i=2}^{s}b_{i}(Z) J \sum_{j=2}^{i-1}a_{ij}(Z) J \sum_{k=2}^{j-1}a_{jk}(Z) J \psi_{2,k}(Z)=0\quad$& 5 \\
12&$\sum_{i=2}^{s}b_{i}(Z) J \sum_{j=2}^{i-1}a_{ij}(Z)c_j K \psi_{2,j}(Z) =0 $& 5 \\
13&$\sum_{i=2}^{s}b_{i}(Z) c_i K \psi_{3,i}(Z)=0 $& 5 \\
14&$\sum_{i=2}^{s}b_{i}(Z) c_i K \sum_{j=2}^{i-1}a_{ij}(Z) J \psi_{2,j}(Z) =0 $& 5 \\
15&$\sum_{i=2}^{s}b_{i}(Z)  B \big(\psi_{2,i}(Z) , \psi_{2,i}(Z) \big)=0 $& 5 \\
16& $\sum_{i=2}^{s}b_{i}(Z) c^2_i L \psi_{2,i}(Z)=0 $& 5 \\
\hline
\end{tabular}
\end{center}
\end{table}
}
%%------------------------------------------------
\subsection{Stability and convergence results}
\label{subsc4.3}
 It was shown in \cite[Theorem 4.1]{LO12b} that the numerical method (\ref{eq2.1}) is stable and converges up to order five, if all conditions of Table~1 are fulfilled.
However, a method satisfying all these conditions will have a plenty of stages. We therefore proceed differently. Motivated by the approach taken in \cite{HO05b}, we relax the order conditions 8--16 in a way described below.

First, we consider the case of constant step size, i.e.~$h_n=h$ for all $n$.
Let $ e_{n+1} = u_{n+1} - u(t_{n+1})=u_{n+1} - \tilde{u}_{n+1}$ denote the global error, i.e., the difference between the numerical and the exact solution of (\ref{eq1.1}), and let $ \hat{e}_{n+1} =u_{n+1} - \hat{u}_{n+1}$ and   $\widehat{E}_{ni}=U_{ni}-\widehat{U}_{ni}$ denote the differences between the two numerical solutions obtained by the schemes (\ref{eq2.1}) and (\ref{eq4.1}), respectively. It is easy to see that the global error at time $t_{n+1}$ satisfies
\begin{equation} \label{eq4.8}
e_{n+1}=\hat{e}_{n+1} + \tilde{e}_{n+1}.
\end{equation}
Subtracting (\ref{eq4.1b}) from (\ref{eq2.1b}) gives
\begin{equation} \label{eq4.9}
\hat{e}_{n+1}=\ee^{h A}e_{n} + h T_{n}
\end{equation}
with
\begin{equation} \label{eq4.10}
T_{n}=\varphi _{1} ( h A) ( g(u_n)-g(\tilde{u}_n)) + \sum_{i=2}^{s}b_{i}(h A)\big(D_{ni}-\ \widehat{D}_{ni}\big).
\end{equation}
Denoting $G_n =g'(\tilde{u}_n)$, \ $G_{ni}=g'( \widehat{U}_{ni} )$ and using the Taylor series expansion of $g(u)$, we get
\begin{subequations} \label{eq4.11}
\begin{align}
g(u_n)-g(\tilde{u}_n)&=G_n e_n + R_n,  \label{eq4.11a} \\
g(U_{ni})-g( \widehat{U}_{ni} )&=G_{ni} \widehat{E}_{ni} + R_{ni} \label{eq4.11b}
\end{align}
\end{subequations}
with remainders $R_n$ and $R_{ni}$
\begin{subequations} \label{eq4.13}
\begin{align}
R_n &=\int_{0}^{1} (1-\theta ) g''(\tilde{u}_n + \theta e_n)(e_n, e_n) \text{d}\theta, \label{eq4.13a}\\
R_{ni} &=\int_{0}^{1} (1-\theta ) g''(\widehat{U}_{ni} + \theta \widehat{E}_{ni})(\widehat{E}_{ni}, \widehat{E}_{ni})
\text{d}\theta. \label{eq4.13b}
\end{align}
\end{subequations}
Employing Assumption~2 shows that
\begin{equation} \label{eq4.12}
\|R_n\|\leq C\|e_n\|^2, \q  \|R_{ni}\|\leq C\|\widehat{E}_{ni}\|^2,
\end{equation}
as long as $e_n$ and $\widehat{E}_{ni}$ remain in a sufficiently small neighborhood of 0.

Subtracting \eqref{eq4.11a} from \eqref{eq4.11b}, we get
\begin{equation} \label{eq4.14}
D_{ni}- \widehat{D}_{ni}=G_{ni} \widehat{E}_{ni} + R_{ni} -(G_n e_n +R_n).
\end{equation}
Inserting \eqref{eq4.11a} and \eqref{eq4.14} into \eqref{eq4.10}, we now obtain
\begin{equation} \label{eq4.15}
T_{n}=b_1 (h A)(G_n e_n +R_n)+\sum_{i=2}^{s} b_i (h A) (G_{ni} \widehat{E}_{ni} + R_{ni})
\end{equation}
with \ $b_1 (hA)=\varphi_1(h A)-\sum_{i=2}^{s} b_i (h A)$.
%%%----LEMMA 4.2------------
\begin{lemma} \label{lm4.2}
Under Assumptions 1 and 2,  there exist bounded operators $\mathcal{K}_{n} (e_n)$ on $X$ such that
\begin{equation}  \label{eq4.16}
T_{n}= \mathcal{K}_{n} (e_n)e_n .
\end{equation}
\end{lemma}
%%---------
 \begin{proof}
Subtracting \eqref{eq4.1a} from \eqref{eq2.1a}, using $\varphi_{1}(z)=(\ee^z -1)/z$ and employing \eqref{eq4.11b}, we obtain
 \begin{eqnarray}   \label{eq4.17}
\widehat{E}_{ni}=U_{ni}-\widehat{U}_{ni}= \ee^{c_i h A}e_{n} + h \sum_{j=1}^{i-1}a_{ij}(hA) (G_{nj} \widehat{E}_{nj} + R_{nj}).
\end{eqnarray}
Solving recursion \eqref{eq4.17}, using \eqref{eq4.13} with an induction argument and inserting the obtained result into \eqref{eq4.15} yields the formula \eqref{eq4.16}.
\end{proof}
In view of \eqref{eq4.8}, \eqref{eq4.9} and \eqref{eq4.16}, we get
\begin{equation}  \label{eq4.18}
e_{n+1}=\ee^{h A}e_{n} + h \mathcal{K}_{n}(e_n)e_n+\tilde{e}_{n+1}.
\end{equation}
Solving recursion \eqref{eq4.18} and using $e_0=0$ finally yields
\begin{equation} \label{eq4.19}
e_{n}=h\sum_{j=0}^{n-1} \ee^{(n-j)hA} \mathcal{K}_j (e_j)e_j + \sum_{j=0}^{n-1} \ee^{jhA}\tilde{e}_{n-j}.
\end{equation}
We are now ready to give a convergence result for constant step sizes that only requires a weakened form of conditions 8--16 of Table \ref{tb1}.
%%--THEOREM 4.2--------
\begin{theorem}\label{th4.1}
Let the initial value problem \eqref{eq1.1} satisfy Assumptions 1--2. Consider for its numerical solution an explicit exponential Runge--Kutta method \eqref{eq2.1} that fulfills the order conditions 1--7 of
Table~\ref{tb1} (up to order four). Further assume that condition 8 holds in a weakened form with $Z=0$ $(\psi_{5} (0)=0)$ and the remaining conditions of order five hold in a weaker form with $b_i(0)$ in place of $b_i (Z)$. Then, the method is convergent of order~5. In particular, the numerical solution $u_n$ satisfies the error bound
\begin{equation}\label{eq4.20}
\| u_n -u(t_n)\|\leq C h^5
\end{equation}
uniformly on compact time intervals \ $t_0 \leq  t_n =t_0+nh \leq  T$ with a constant $C$ that depends on $T-t_0$, but is independent of $n$ and $h$.
\end{theorem}
\begin{proof}
In view of \eqref{eq4.6a} and \eqref{eq4.5}, under the assumptions of Theorem~\ref{th4.1}, we obtain
\begin{equation} \label{eq4.21}
\tilde{e}_{n+1}=h^5 \big( \psi_{5} ( h A)-\psi_{5} (0)\big)\mathbf{P}_n+h^5 \sum_{i=2}^{s} \big(b_i (h A)-b_i (0)\big)\mathbf{Q}_n+h^6 \mathbf{S}_n,
\end{equation}
where $\mathbf{Q}_n$ denotes the terms multiplying $h^5 \sum_{i=2}^{s} b_i $ in \eqref{eq4.6a} and $\|\mathbf{S}_n\|\leq C.$
Inserting \eqref{eq4.21} (with index $n-j-1$ in place of $n$) into \eqref{eq4.19} yields
\begin{equation} \label{eq4.22}
\begin{aligned}
e_{n}&=h\sum_{j=0}^{n-1} \ee^{(n-j)hA} \mathcal{K}_j (e_j)e_j + h^5\sum_{j=0}^{n-1} \ee^{jhA}\big( \psi_{5} ( h A)-\psi_{5} (0)\big)\mathbf{P}_{n-j-1} \\
&\quad + h^5\sum_{j=0}^{n-1} \ee^{jhA} \sum_{i=2}^{s} \big(b_i (h A)-b_i (0)\big)\mathbf{Q}_{n-j-1}  +h^6 \mathbf{S}_{n-j-1}.
\end{aligned}
\end{equation}
We can now employ the same techniques that were used in the proof of \cite[Theorem 4.7]{HO05b}.
The bound \eqref{eq4.20} follows from the stability bound \eqref{eq3.2}, the bound \cite[Lemma 4.8]{HO05b} and an application of a discrete Gronwall lemma.
\end{proof}
Under further assumptions on the regularity of the solution or on the step size sequence, it is possible to generalize Theorem~\ref{th4.1} to variable step sizes. We omit the details.
%SECTION 5 -----------------------------------
\section{Construction of methods of order 5} \label{sc5}
In this section, we will construct exponential Runge--Kutta methods of order five based on the weakened form of the order conditions 8--16, as mentioned in Theorem \ref{th4.1}. In fact, one needs to solve a system of 16 equations (corresponding to 16 order conditions of Table \ref{tb1}) for unknown coefficients $b_i, a_{ij}.$ The obtained coefficients of the method will be displayed in the following Butcher tableau:
\begin{displaymath}
\renewcommand{\arraystretch}{1.1}
\begin{array}{c|ccccc}
c_2&&&&\\
c_3&a_{32}&&&\\
\vdots&\vdots&\ddots&&\\
c_s&a_{s2}&\ldots&a_{s-1,s}&\\
\hline
&b_2&\ldots&b_{s-1}&b_s
\end{array}
\end{displaymath}
We are interested in finding methods with a minimal number of stages. In order to answer this question, we first state the following two technical lemmas.
%%---LEMMA 5.1----
\begin{lemma} \label{lm5.1}
Let $X_1 (z), X_2 (z),\ldots, X_\ell (z)$ be linearly independent (analytic) functions and assume that for all square matrices $Z, J$ of the same format
\begin{equation} \label{eq4.23a}
\sum_{i=1}^{\ell} X_i (Z) J Y_i (Z) = 0
\end{equation}
for given (analytic) functions $Y_i (z)$. Then $Y_i=0$ for all $i=1,2,\ldots,\ell$.
\end{lemma}
\begin{proof}
Let
$
 Z=\begin{bmatrix}
\lambda &0\\
  0&\mu
\end{bmatrix}
\text{and} \
J=\begin{bmatrix}
  0&1\\
  1&0
\end{bmatrix}.~
$
Inserting these matrices into \eqref{eq4.23a} shows that $\sum_{i=1}^{\ell} X_i (\lambda ) Y_i (\mu ) = 0$ for all $\lambda, \mu $. Since the set $\{X_i (z)\}_{i=1}^{\ell}$ is linearly independent, we get $Y_i (\mu )=0$ for all $ \mu $, which proves the lemma.
\end{proof}
%%---LEMMA 5.2----
\begin{lemma} \label{lm5.2}
Let $X_1 (z), X_2 (z),\ldots, X_\ell (z)$ be linearly independent analytic functions and assume that for all square matrices $Z, J$ of the same format
\begin{equation} \label{eq4.23b}
\sum_{i=1}^{\ell} X_i (Z) J \sum_{j=1}^{m}Y_{ij} (Z)J W_j (Z) = 0
\end{equation}
for given (analytic) functions $Y_{ij} (z), W_j (z) $. Then, for all $i=1,2,\ldots,\ell$, we have
\begin{equation} \label{eq4.23c}
\sum_{j=1}^{m}Y_{ij} (\mu) W_j (\nu ) = 0.
\end{equation}
for all $\mu, \nu$.
\end{lemma}
\begin{proof}
We choose the matrices
\begin{equation*}
 Z=\begin{bmatrix}
\lambda &0 & 0 \\
  0&\mu & 0 \\
  0& 0& \nu
\end{bmatrix}, \quad
J=\begin{bmatrix}
  0&1&0\\
  0&0&1 \\
  1&0&0
\end{bmatrix}
\end{equation*}
and insert them into \eqref{eq4.23b}. This shows that
$\sum_{i=1}^{\ell} X_i (\lambda )\sum_{j=1}^{m}Y_{ij} (\mu ) W_j (\nu ) = 0$ for all $\lambda, \mu, \nu $.
Since the set $\{X_i (z)\}_{i=1}^{\ell}$ is linearly independent, we obtain
$\sum_{j=1}^{m} Y_{ij} (\mu ) W_j (\nu )= 0 $ for all $\mu, \nu $.
\end{proof}
With the results of Lemmas \ref{lm5.1} and \ref{lm5.2} at hand, we are now in the position to state the main result of this section.
%%--THEOREM 4.3--------
\begin{theorem}\label{th5.1}
There does not exist an explicit exponential Runge--Kutta method of order~5 with less than or equal to 6~stages.
\end{theorem}
\begin{proof}
Note that scheme \eqref{eq2.1} reduces to a classical Runge--Kutta method for $A=0$. A famous result by Butcher shows that an explicit Runge--Kutta method of order five  with $s=5$ stages does not exist (see \cite{Butcher64}). We deduce from this result that, for $s=5$, no explicit exponential Runge--Kutta method of order 5 exists.
Therefore, we now consider the case $s=6$. Based on Theorem \ref{th4.1}, one needs to check conditions 1--7 of Table \ref{tb1}. For $s=6$, conditions 1, 2 and 4 are
\begin{subequations} \label{eq4.24}
\begin{align}
b_2 c_2+ b_3 c_3+ b_4 c_4+b_5 c_5+b_6 c_6 &=\varphi_2, \\
b_2 c^2_2+ b_3 c^2_3+ b_4 c^2_4+b_5 c^2_5+b_6 c^2_6&=2\varphi_3,\\
b_2 c^3_2+ b_3 c^3_3+ b_4 c^3_4+b_5 c^3_5+b_6 c^3_6&=6\varphi_4,
\end{align}
\end{subequations}
 and conditions 3, 5, 7 and 6 read
\begin{subequations} \label{eq4.25}
\begin{align}
b_2 J \psi_{2,2}+ b_3 J \psi_{2,3}+ b_4 J \psi_{2,4}+ b_5 J \psi_{2,5}+ b_6 J \psi_{2,6}&= 0, \label{eq4.25a} \\
b_2 J \psi_{3,2}+ b_3 J \psi_{3,3}+ b_4 J \psi_{3,4}+ b_5 J \psi_{3,5}+ b_6 J \psi_{3,6}&= 0, \label{eq4.25b} \\
b_2 c_2 K \psi_{2,2}+ b_3 c_3 K \psi_{2,3}+ b_4 c_4 K \psi_{2,4}+ b_5 c_5 K \psi_{2,5}+ b_6 c_6 K \psi_{2,6}&=0 \label{eq4.25c},
\end{align}
\end{subequations}
and
\begin{equation}\label{eq4.26}
\begin{aligned}
b_3 J a_{32}J\psi_{2,2} & + b_4 J (a_{42}J\psi_{2,2}+a_{43}J\psi_{2,3}) \\
& +b_5 J(a_{52}J\psi_{2,2}+a_{53}J\psi_{2,3}+a_{54}J\psi_{2,4})\\
& +b_6J(a_{62}J\psi_{2,2}+a_{63}J\psi_{2,3}+a_{64}J\psi_{2,4}+ a_{65}J\psi_{2,5})=0,
\end{aligned}
\end{equation}
respectively. From \eqref{eq4.4}, one derives
\begin{equation}\label{eq4.27}
\psi_{2,i} = \sum_{j=2}^{i-1}a_{ij}c_j-c^{2}_i \varphi_{2,i}, \q
\psi_{3,i} = \sum_{j=2}^{i-1}a_{ij}\frac{c^{2}_j}{2!}-c^{3}_i \varphi_{3,i}.
\end{equation}
We note for later use that $c_2\ne 0$ (otherwise $U_{n2}=u_n$ and we are back to the case of 5 stages) and
\begin{equation}\label{eq4.27a}
 \psi_{2,2}=-c^{2}_{2}\varphi_{2,2} \ne 0, \q \psi_{3,2}=-c^{3}_{2}\varphi_{3,2} \ne 0.
\end{equation}

%%----Figure 1------------------------
\begin{figure}[b]
\begin{center}
\includegraphics[scale=0.65]{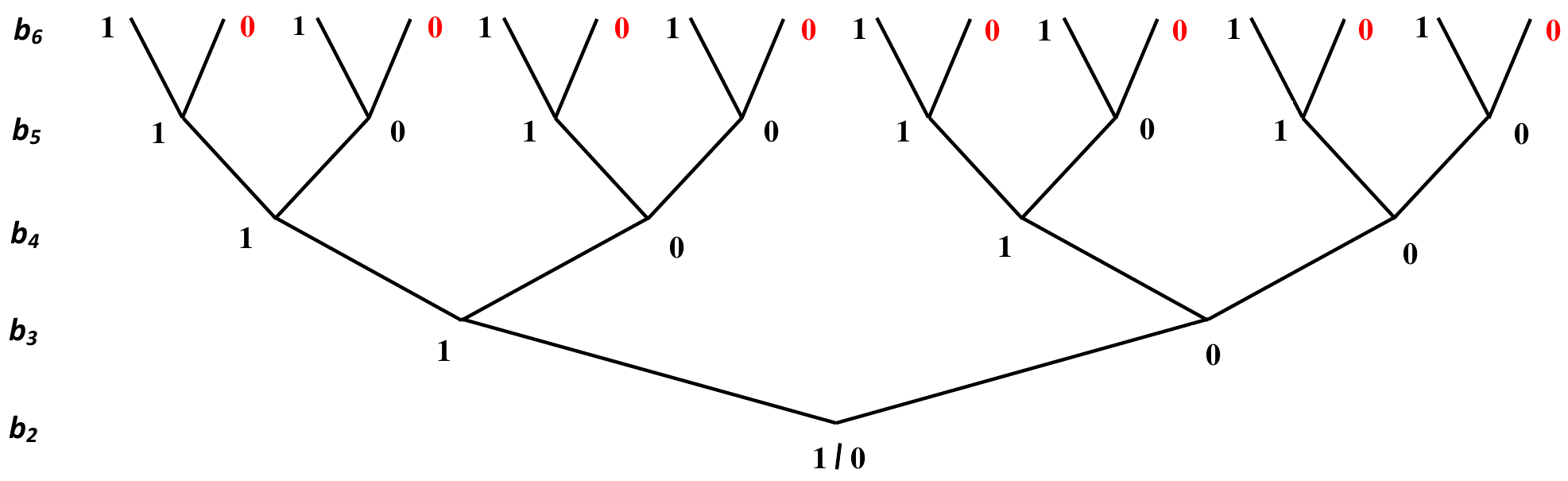}
\end{center}
\caption{\label{fig1} The full binary tree $b_2 b_3 b_4 b_5 b_6$. }
\end{figure}
%%%-----------------------------------

For the subsequent considerations, it turns out to be important whether the coefficients $b_i$ are zero or not. We visualize all occurring possibilities by a full binary tree, see Figure 1. Note that depending on the labeling of $b_2$, Figure 1 actually represents two different trees. Every node of the tree is labeled by the binary values 0, 1 (where 0 refers to $b_i = 0$ and 1 refers to $b_i \ne 0$). The tree begins with a root which is considered as level 1. The children of the root are at level 2, their children are at the next level, and so on. Thus, there are five levels from the root to the leaves of the tree. Motivated by the interpretation, it is convenient to denote the levels 1--5 of the tree by $b_2, b_3, b_4, b_5, b_6,$ respectively (see Figure 1). Subtrees from level $i$ to $j$ will be denoted by $b_i b_{i+1}\ldots b_j$.
We consider the following two cases.

\emph{Case I: } $b_2 \ne 0$ (the root of the tree is now labeled by 1). \\
For fixed $b_2$, there are $2^4=16$ ways of labeling $b_3, b_4, b_5, b_6$ by binary values. However, when $b_6$ is labeled by 0, i.e.~$b_6=0$, we are back to the case $s=5$ which was treated at the beginning of the proof. Therefore, $b_6$ has to be labeled by 1 ($b_6 \ne 0$). It is then easy to exclude the case $b_3 =b_4= b_5 =0$ since the linear system of equations \eqref{eq4.24} has no solution. In total, we are left with the following seven possibilities for labeling $b_2,\ldots,b_6$:  1) 10011; 2) 10101; 3) 11001; 4) 10111; 5) 11011; 6) 11101 and 7) 11111 (see the corresponding paths in Figure 1).

{\em Subcase} 1): 10011 ($b_3=b_4=0,  b_2,  b_5,  b_6 \ne 0 $). In this case, the determinant of the linear system of equations \eqref{eq4.24} is $c_2 c_5 c_6 (c_2 - c_5)(c_2 - c_6) (c_6-c_5)$. If it is zero then the system has no solution due to the fact that the functions $\varphi_2, \varphi_3,  \varphi_4$ are linearly independent.  Therefore, we deduce that the nodes $c_2, c_5$ and $c_6$ must be distinct. Consequently, the system has a unique solution for $b_2,  b_5,  b_6 $. It is also clear that $b_2,  b_5,  b_6$ are linearly independent. However, \eqref{eq4.25a} is then in contradiction with \eqref{eq4.27a} due to Lemma~\ref{lm5.1}.

{\em Subcases} 2) and 3): These two cases can be treated in exactly the same way as subcase 1). We eventually end up with the same contradiction to~\eqref{eq4.27a}.

{\em Subcase} 4): 10111 ($b_3=0,  b_2,  b_4,  b_5,  b_6 \ne 0 $).
To satisfy \eqref{eq4.25a}, the functions $b_2, b_4,  b_5,  b_6$ must be linearly dependent due to Lemma \ref{lm5.1} and~\eqref{eq4.27a}. Hence, there exist scalars $\alpha, \beta, \gamma $, not all zero, such that we have one of the following four possibilities:  (i) $b_2=\alpha b_4+\beta b_5+\gamma b_6$ or (ii) $b_4=\alpha b_2+\beta b_5+\gamma b_6$ or (iii) $b_5=\alpha b_2+\beta b_4+\gamma b_6$ or (iv) $b_6=\alpha b_2+\beta b_4+\gamma b_5$.
Inserting one of these representations into \eqref{eq4.24}, we obtain a system of three linear equations in three unknowns $b_i$ among $ b_2, b_4,  b_5, b_6$. This system cannot have infinitely many solutions since $\varphi_2, \varphi_3,  \varphi_4$ are linearly independent. Moreover, if it has a solution, then the solution is unique and thus the three unknowns $b_i$ are linearly independent. For the sake of presentation, we shall here only treat case (i) in detail (the remaining cases (ii)--(iv) are treated in a similar way). We are thus given the linear system
\begin{equation}\label{eq4.27add}
\begin{aligned}
(\alpha c_2 +c_4)b_4+(\beta c_2+c_5)b_5+(\gamma c_2+c_6)b_6&=\varphi_2   \\
(\alpha c^2_2 +c^2_4)b_4+(\beta c^2_2+c^2_5)b_5+(\gamma c^2_2+c^2_6)b_6&=2\varphi_3   \\
(\alpha c^3_2 +c^3_4)b_4+(\beta c^3_2+c^3_5)b_5+(\gamma c^3_2+c^3_6)b_6&=6\varphi_4
\end{aligned}
\end{equation}
with the three unknowns $b_4,  b_5, b_6$. Inserting the given representation of $b_2$ into \eqref{eq4.25a} and \eqref{eq4.25b} gives
\begin{equation} \label{eq4.27add1}
\begin{aligned}
b_4 J (\alpha \psi_{2,2}+\psi_{2,4})+ b_5 J (\beta \psi_{2,2}+ \psi_{2,5})+ b_6 J(\gamma \psi_{2,2}+ \psi_{2,6})&= 0, \\
b_4 J (\alpha \psi_{3,2}+\psi_{3,4})+ b_5 J (\beta \psi_{3,2}+ \psi_{3,5})+ b_6 J(\gamma \psi_{3,2}+ \psi_{3,6})&= 0.
\end{aligned}
\end{equation}
By applying Lemma~\ref{lm5.1} to \eqref{eq4.27add1}, we derive the relations $\psi_{2,4}=-\alpha \psi_{2,2}$ and $\psi_{3,4}=-\alpha  \psi_{3,2}$. Using \eqref{eq4.27}, we get
\begin{equation}\label{eq4.27add3}
\begin{aligned}
a_{42}c_2 + a_{43}c_3&=c^2_4 \varphi_{2,4}+\alpha c^2_2 \varphi_{2,2},  \\
a_{42}c^2_2 + a_{43}c^2_3&=2c^3_4 \varphi_{3,4}+2\alpha c^3_2 \varphi_{3,2}.
\end{aligned}
\end{equation}
If $c_2=c_3$, system \eqref{eq4.27add3} has either no solution or infinitely many solutions. It is easy to check that the latter occurs only in the case $c_2=c_4$ and $\alpha=-1$ (giving $a_{42}+a_{43}=0$). In that case, however, system  \eqref{eq4.27add} has no solution. If $c_2 \ne c_3$, system \eqref{eq4.27add3} has an unique solution for $a_{42}, a_{43}$. One can check that they are linearly independent if either $c_2 \ne c_4$ or $c_2=c_4$ and $\alpha \ne -1$. In the remaining case $c_2=c_4$ and $\alpha=-1$, the solution is $a_{42}=a_{43}=0$. But, again this gives a contradiction with the fact that \eqref{eq4.27add} has an unique solution.
We now insert the relations $\psi_{2,4}=-\alpha \psi_{2,2}$ and $\psi_{2,5}=-\beta  \psi_{2,2}$ (both follow from \eqref{eq4.27add1}) into \eqref{eq4.26} to get
\begin{equation}\label{eq4.27b}
\begin{aligned}
 & b_4 J (a_{42}J\psi_{2,2}+a_{43}J\psi_{2,3})+ b_5 J\big((a_{52}-\alpha a_{54})J\psi_{2,2}+a_{53}J\psi_{2,3}\big)\\
&\q +b_6J\big((a_{62}-\alpha a_{64}-\beta a_{65})J\psi_{2,2}+a_{63}J\psi_{2,3} \big)=0.
\end{aligned}
\end{equation}
Applying Lemma \ref{lm5.2} to \eqref{eq4.27b} shows that $a_{42}(\mu)\psi_{2,2}(\nu)+ a_{43}(\mu)\psi_{2,3}(\nu)=0$
for all complex $\mu, \nu$. Since $\psi_{2,2}\ne 0$, we get a contradiction with the fact that the functions $a_{42}, a_{43}$ are linearly independent.

{\em Subcases} 5) and 6): We have either $b_4=0$ or $b_5=0$, whereas all other $b_i \ne 0 $. This case is even simpler than subcase 4). From conditions \eqref{eq4.25a} and \eqref{eq4.25b} we deduce in exactly the same way as in subcase 4) that
\begin{equation}\label{eq4.27c}
\psi_{2,3}=\kappa   \psi_{2,2} \q \text{and} \q \psi_{3,3}=\kappa   \psi_{3,2}
\end{equation}
hold with some scalar $\kappa$. For instance, if $b_3, b_5, b_6$ are linearly independent and $b_2=\alpha b_3+\beta b_5+ \gamma b_6$, one derives $\alpha \psi_{2,3}+\psi_{2,2}=\alpha \psi_{3,3}+\psi_{3,2}=0$ which implies that $\alpha \ne 0$ due to \eqref{eq4.27a}. Thus, \eqref{eq4.27c} holds with $\kappa=-\frac{1}{\alpha}$. Now, if either $c_2\ne c_3$ or $c_2= c_3$ and $\kappa \ne 1$, then \eqref{eq4.27c} contradicts \eqref{eq4.27}. In the remaining case $c_2= c_3$ and $\kappa = 1$ (implying $a_{32}=0$), system \eqref{eq4.24} has no solution.

{\em Subcase} 7): 11111 ($b_2,  b_3,  b_4,  b_5,  b_6 \ne 0 $).
Again, an application of Lemma~\ref{lm5.1} to condition \eqref{eq4.25a} shows that $ b_2, b_3, b_4, b_5, b_6$ are linearly dependent. Assume for a moment that four of them are linearly independent and express the fifth weight in terms of the others. This gives five different cases which all eventually lead to condition \eqref{eq4.27c}. In addition, one deduces
\begin{equation}\label{eq4.27d}
\psi_{2,4}=\eta \psi_{2,2} \q \text{and} \q \psi_{3,4}=\eta   \psi_{3,2}
\end{equation}
with some scalar $\eta$.
(For instance, assume that $b_3, b_4, b_5, b_6$ are linearly independent and $b_2=\alpha b_3+\beta b_4+\gamma b_5+ \delta b_6$. Inserting this representation into \eqref{eq4.25a}, \eqref{eq4.25b} and applying once more Lemma \ref{lm5.1} shows that \eqref{eq4.27c} and \eqref{eq4.27d} hold with $\kappa=-\alpha$ and $\eta=-\beta,$ respectively.) As we have shown in subcases 5) and 6), \eqref{eq4.27c} is only satisfied if $c_2= c_3$ and $\kappa = 1$. However, this is not sufficient to conclude that the system \eqref{eq4.24} has no solution in this case. Fortunately, under the condition that $c_2= c_3$, \eqref{eq4.27d} is satisfied only if $c_2= c_4$ and $\eta=1$ (similar to the consideration of \eqref{eq4.27add3} with $-\eta$ in place of $\alpha$). As a consequence, $b_3, b_4, b_5, b_6$ must be linearly dependent. The remaining cases are treated in a similar way. This shows that at most three weights $b_i$ are linearly independent. Therefore, ten combinations have to be checked. Again, we just detail here one typical case. Assume that there exist scalars $\alpha_j , \beta_j ,\gamma_j \ (j=1,2)$, not all zero, such that $b_2=\alpha_1 b_4+\beta_1 b_5+\gamma_1 b_6, \ b_3=\alpha_2 b_4+\beta_2 b_5+\gamma_2 b_6$.
We insert these expressions into \eqref{eq4.24}. If the resulting system has a solution, then it is unique and $b_4, b_5, b_6$ are linearly independent. Note that in this case the first and the second column of the coefficient matrix of this system, denoted by $\Delta_1$ and $\Delta_2$, respectively, are given by
\begin{subequations}\label{eq4.27e}
\begin{align}
\Delta_1&=\left[\alpha_1 c_2+\alpha_2 c_3+ c_4, \ \alpha_1 c^2_2+\alpha_2 c^2_3+ c^2_4, \ \alpha_1 c^3_2+\alpha_2 c^3_3+ c^3_4\right]^{\sf T},\\
\Delta_2&=\left[\beta_1 c_2+\beta_2 c_3+ c_5, \ \beta_1 c^2_2+\beta_2 c^2_3+ c^2_5, \ \beta_1 c^3_2+\beta_2 c^3_3+ c^3_5\right]^{\sf T} \label{eq4.27e2}.
\end{align}
\end{subequations}
An application of Lemma \ref{lm5.1} tells us that conditions \eqref{eq4.25a} and \eqref{eq4.25c} are satisfied if and only if the following conditions hold
\begin{subequations} \label{eq4.28}
\begin{align}
\alpha_2 \psi_{2,3}+\psi_{2,4} =&-\alpha_1 \psi_{2,2}, \label{eq4.28a} \\
c_3 \alpha_2 \psi_{2,3}+c_4 \psi_{2,4} =&-c_2 \alpha_1 \psi_{2,2}, \label{eq4.28b} \\
\beta_2 \psi_{2,3}+\psi_{2,5} =&-\beta_1 \psi_{2,2}, \label{eq4.28c} \\
c_3 \beta_2 \psi_{2,3}+c_5 \psi_{2,5} =&-c_2 \beta_1 \psi_{2,2}, \label{eq4.28d}\\
\gamma_2 \psi_{2,3}+\psi_{2,6} =&-\gamma_1 \psi_{2,2}, \label{eq4.28e}\\
c_3 \gamma_2 \psi_{2,3}+c_6 \psi_{2,6} =&-c_2 \gamma_1 \psi_{2,2}. \label{eq4.28f}
\end{align}
\end{subequations}
This is a linear system with unknowns $\psi_{2,i} \ (i=3,4,5,6)$. We distinguish two cases. First, if \eqref{eq4.28} has a unique solution, then $\psi_{2,i}=\kappa_i\psi_{2,2} \ (i=3,4,5,6)$ with some scalars $\kappa_i$.
Since $\psi_{2,3}=\kappa_3\psi_{2,2}$, we get $a_{32}=(c^2_3\varphi_{2,3}-\kappa_3 c^2_2 \varphi_{2,2})/c_2 $. From this, one derives $\psi_{3,3}$ by using \eqref{eq4.27}. We then note that for satisfying condition \eqref{eq4.25b}, it is necessary to have $\psi_{3,4} =-\alpha_1 \psi_{3,2}-\alpha_2 \psi_{3,3}$ which allows us to compute $\psi_{3,4}$. Knowing $\psi_{2,4}=\kappa_4 \psi_{2,2}$ and $\psi_{3,4}$, one obtains the following system of two linear equations
\begin{subequations}\label{eq4.29}
\begin{align}
a_{42}c_2 + a_{43}c_3&=c^2_4 \varphi_{2,4}-\kappa_4 c^2_2 \varphi_{2,2} \label{eq4.29a} \\
a_{42}c^2_2 + a_{43}c^2_3&=2c^3_4 \varphi_{3,4}+2\alpha_1 c^3_2 \varphi_{3,2}-\alpha_2 a_{32}c^2_2+ 2\alpha_2 c^3_3 \varphi_{3,3} \label{eq4.29b}
\end{align}
\end{subequations}
with two unknowns $a_{42}, a_{43}$. We consider two cases: either \eqref{eq4.29} has a unique solution or it has infinitely many solutions. The latter occurs if and only if $c_2=c_3=c_4$ and $1+(1-\kappa_3)\alpha_2-\kappa_4=1+\alpha_1+\alpha_2=0$. In view of \eqref{eq4.27e}, we get $\Delta_1=(1+\alpha_1+\alpha_2)[c_2, c^2_2, c^3_2\,]^{\sf T}=[0, 0, 0]^{\sf T}$. This immediately shows that~\eqref{eq4.24} has no solution. On the other hand, if \eqref{eq4.29} has a unique solution then $c_2 \ne c_3$. We now insert the (unique) solution of \eqref{eq4.28} into \eqref{eq4.26} to get
$b_4 J (\alpha_2 a_{32}+ a_{42} +\kappa_3 a_{43} ) + b_5 J(\beta_2 a_{32}+ a_{52}+\kappa_3 a_{53}+\kappa_4 a_{54})+ b_6 J(\ldots)=0$. This implies that
\begin{subequations}\label{eq4.29c}
\begin{align}
\alpha_2 a_{32}+ a_{42} +\kappa_3 a_{43}&=0, \label{eq4.29c1} \\
\beta_2 a_{32}+ a_{52}+\kappa_3 a_{53}+\kappa_4 a_{54}&=0 \label{eq4.29c2}
\end{align}
\end{subequations}
because of Lemma~\ref{lm5.1}.
From \eqref{eq4.29}, we compute $a_{42}, a_{43}$ explicitly (by using the known function $a_{32}$). Inserting the values into \eqref{eq4.29c1} gives
\begin{multline*}
\alpha_2 c^2_3 (c^2_3-\kappa_3 c^2_2)\varphi_{2,3}-c^2_2(\kappa_4+\alpha_2 \kappa_3)(c^2_3-\kappa_3 c^2_2)\varphi_{2,2}+c^2_4 (c^2_3-\kappa_3 c^2_2)\varphi_{2,4}\\
+2c^3_4 (\kappa_3 c_2-c_3)\varphi_{3,4}+2\alpha_1c^3_2 (\kappa_3 c_2-c_3)\varphi_{3,2}+2\alpha_2 c^3_3 (\kappa_3 c_2-c_3)\varphi_{3,3}=0.
\end{multline*}
It is easy to check that this condition is satisfied if and only if one of the following six cases occurs: (i) $c_2=c_4$, $\kappa_3=\frac{c_3}{c_2}$, $\kappa_4=1$, $\alpha_2=0$, or (ii) $c_2=c_4$, $\kappa_3=\frac{c^2_3}{c^2_2}$, $\alpha_1=-1$, $\alpha_2=0$, or (iii) $c_2=c_4$, $\kappa_4=1$, $\alpha_1=-1$, $\alpha_2=0$, or (iv) $c_3=c_4$, $\kappa_3 = \kappa_4=\frac{c_3}{c_2}$, $\alpha_2 = -1$, or (v) $c_3=c_4$, $\kappa_3 = \kappa_4$, $\alpha_1 = 0$, $\alpha_2 = -1$, or (vi) $c_3=c_4$, $\kappa_3 = \frac{c^2_3}{c^2_2}$, $\alpha_1 = 0$, $\alpha_2 = -1$. We note that only the cases (i) and (iv) have to be investigated because the remaining ones lead to $\Delta_1=[0, 0, 0]^{\sf T}$ which gives the same contradiction as above.

We first treat case (i). Inserting the known values into \eqref{eq4.29c2} and using the fact that $\psi_{2,5}=\kappa_5\psi_{2,2}$, one obtains
\begin{equation}\label{eq4.29d}
\beta_2 c^2_3 \varphi_{2,3}-c_2 (\beta_2 c_3+ \kappa_5 c_2) \varphi_{2,2}+c^2_5 \varphi_{2,5}=0.
\end{equation}
This condition is satisfied if either $c_2=c_5$, $\beta_2=0$, $\kappa_5=1$ or $c_3=c_5$, $\beta_2=-1$, $\kappa_5=\frac{c_3}{c_2}$. On the one hand, in view of \eqref{eq4.27e2}, these two constraints give $\Delta_2=(1+\beta_1)[c_2, c^2_2, c^3_2\,]^{\sf T}$ and $\Delta_2=\beta_1[c_2, c^2_2, c^3_2\,]^{\sf T}$, respectively. On the other hand, the given relations of case (i) yield $\Delta_1=(1+\alpha_1)[c_2, c^2_2, c^3_2\,]^{\sf T}$ which again shows that the determinant of the resulting system in three unknowns $b_4, b_5, b_6$ is zero. Following the same strategy in case (iv) gives again condition \eqref{eq4.29d}. Note that in this case $\Delta_1=\alpha_1[c_2, c^2_2, c^3_2\,]^{\sf T}$ which yields the same contradiction as in case~(i).

Second, if system \eqref{eq4.28} has infinitely many solutions, then we deduce that two of the nodes among $c_4, c_5, c_6$ are equal. This immediately shows that system \eqref{eq4.24} has no solution.

{\em Case II: } $b_2 = 0$ (the root of the tree is now labeled by 0). \\
In this case, we can remove level 1 from the tree and consider $b_3b_4b_5b_6$. We now distinguish two cases.

{\em Case II.1:} $b_3 \ne 0$ (see the left subtree in Figure 1). By carrying out the same procedure as in case I, one has to consider the following three subcases: 1) 1011; 2) 1101 and 3) 1111.

{\em Subcases} 1) and 2): By assumption, there exists $i \in \{4,5 \}$ with $b_i \ne 0$. From this, we deduce that the nodes $c_3, c_6, c_i$ are distinct, otherwise we get a contradiction to \eqref{eq4.24}. We therefore obtain a unique solution for $b_3,  b_6,  b_i \ne 0 $. An application of Lemma~\ref{lm5.1} shows that the conditions \eqref{eq4.25a} and \eqref{eq4.25b} are satisfied if $ \psi_{2,3}=\psi_{3,3}=0$. But this contradicts \eqref{eq4.27}.

{\em Subcase} 3): 1111 ($ b_3,  b_4,  b_5,  b_6 \ne 0 $). The four weights cannot be linearly independent, otherwise we get from \eqref{eq4.25} that $\psi_{2,3}=\psi_{3,3}=0$ which contradicts \eqref{eq4.27}. We thus have to study the four cases: $b_3=\alpha b_4+\beta b_5+\gamma b_6$ or $b_4=\alpha b_3+\beta b_5+\gamma b_6$ or $b_5=\alpha b_3+\beta b_4+\gamma b_6$ or $b_6=\alpha b_3+\beta b_4+\gamma b_5$. We exemplify again the first case.
To satisfy \eqref{eq4.25a} and \eqref{eq4.25b}, it is necessary to have
\begin{equation}\label{eq4.30}
\psi_{2,4}=-\alpha \psi_{2,3} \ \text{and} \ \psi_{3,4}=-\alpha \psi_{3,3}.
\end{equation}
% We note that $\psi_{2,3}\ne 0$, since otherwise $\psi_{2,4}=\psi_{2,5}=\psi_{2,6}=0$ and thus $\psi_{3,3}\ne 0$ which contradicts \eqref{eq4.27}.
From \eqref{eq4.26} we deduce that $(a_{42}(z) +\alpha a_{32}(z))\psi_{2,2}(\mu) + a_{43}(z)\psi_{2,3}(\mu)=0$. This shows that $a_{42}+\alpha a_{32}$ and $a_{43}$ are linearly dependent. Using this, we derive that \eqref{eq4.30} is satisfied if and only if $\alpha =-1$ and $c_3=c_4$. However, this gives a contradiction with the solvability of \eqref{eq4.24}.

{\em Case II.2:} $b_3= 0$ (see the right subtree in Figure 1). We can now remove level 2 from the tree and consider $b_4b_5b_6$ only. This case can be easily carried out by the same techniques to show that condition \eqref{eq4.26} cannot be satisfied.

Altogether, we have shown that, for $s=6$, the method even does not satisfy the stiff order conditions 1--7 (order four) in the strong sense.
\end{proof}

The case $s=7$ can be analyzed in a similar way by adding an additional level for $b_7$ in Figure 1. The actual computations, however, are much more tedious to carry out. In particular, it is no longer true that the order conditions 1--7 already give a contradiction. On the other hand, 8 stages are enough to construct a method of order~5 as will be shown now.

We start off with the choice $b_2=b_3=b_4=b_5=0$ that significantly simplifies the solution of the order conditions. From conditions 1, 2 and 4, we obtain
\begin{equation*}
\begin{aligned}
 b_6 &=\dfrac{6\varphi_4 -2(c_7 +c_8)\varphi_3 + c_7 c_8 \varphi_2}{c_6(c_6-c_7)(c_6-c_8)}, \\
b_7&=\dfrac{-6\varphi_4 +2(c_6 +c_8)\varphi_3 - c_6 c_8 \varphi_2}{c_7(c_6-c_7)(c_7-c_8)}, \\
b_8&=\dfrac{6\varphi_4 -2(c_6 +c_7)\varphi_3 + c_6 c_7 \varphi_3}{c_8(c_6-c_8)(c_7-c_8)}
\end{aligned}
\end{equation*}
for any choice of distinct (positive) nodes $c_6, c_7, c_8$. Inserting these values (evaluated at $Z=0$) into condition 8 and choosing $c_8=1$, we get the following algebraic equation
\[ 10c^2_6 c^2_7 - (c_6+c_7)(15c_6 c_7+8)+5(c^2_6+c^2_7 )+23c_6 c_7+3=0, \]
which is equivalent to $ 10c_6 c_7=5(c_6+c_7)-3$ if $c_6 \ne 1$ and $c_7 \ne 1$. To solve it, we take $c_6=\frac{1}{5}$ and get $c_7=\frac{2}{3}$. To satisfy conditions 3, 5 and 7, we must enforce $\psi_{2,i}=\psi_{3,i}=0 ~(i=6, 7, 8)$.  Thus, conditions 13, 15 and 16 are satisfied automatically. Next, we choose $a_{6k}=a_{7k}=a_{8k}=0 ~ (k=2,3)$ and enforce $\psi_{2,4}=\psi_{2,5}=0$ in order to satisfy condition 6. It is now easy to see that conditions 12 and 14 are satisfied. As a compromise between conditions 10 and 11, we require~$b_6(0)a_{64}+ b_7(0)a_{74}+b_8(0)a_{84}=0, \ a_{52}=0, \ \psi_{3,5}=\psi_{2,3}=0$ for satisfying both of them. Finally, condition 9 requires
$b_6(0)\psi_{4,6}+ b_7(0)\psi_{4,7}+b_8(0)\psi_{4,8}=0$.
Collecting all the requirements, we obtain a system of eleven linear equations in thirteen unknowns $a_{ij}$ (except the already known values). The coefficients $a_{42}, a_{84}$ can be taken as free parameters. In the following we choose $c_2=c_3=c_5=\frac{1}{2},~ c_4=\frac{1}{4},~ a_{42}=a_{84}=0$. This yields the following fifth-order scheme which will be called $\mathtt{expRK5s8}$:

\begin{displaymath}
\renewcommand{\arraystretch}{1.4}
\begin{tabular}{c|ccccccc}
$\frac{1}{2}$~&&&&&&&\\\
$\frac{1}{2}$ \ &$\frac{1}{2}\varphi_{2,3}$&&& &&&\\
$\frac{1}{4}$& 0 & $\frac{1}{8}\varphi_{2,4}$&& &&&\\
$\frac{1}{2}$& 0 & $-\frac{1}{2}\varphi_{2,5}+2\varphi_{3,5}$ & $2\varphi_{2,5}-4\varphi_{3,5}$& &&& \\
$\frac{1}{5}$& 0 & 0 & $a_{64}$ & $\frac{2}{25}\varphi_{2,6}-\frac{1}{2}a_{64}$ &&& \\
$\frac{2}{3}$& 0 & 0 & $-\frac{125}{162}a_{64}$& $a_{75}$ & $a_{76}$ && \\
$1$& 0 & 0 & 0& $a_{85} $ & $a_{86}$ & $a_{87}$ \\[2pt]
\hline
& 0 & 0& 0 & 0 & $b_6$ & $b_7$ &  $b_8$
\end{tabular}
\end{displaymath}
with
\begin{equation*}
\begin{aligned}
a_{64}&=\tfrac{8}{25}\varphi_{2,6}-\tfrac{32}{125}\varphi_{3,6},\\
a_{75}&=\tfrac{125}{1944} a_{64}-\tfrac{16}{27}\varphi_{2,7}+\tfrac{320}{81}\varphi_{3,7}, \qquad
a_{76}=\tfrac{3125}{3888} a_{64}+\tfrac{100}{27}\varphi_{2,7} -\tfrac{800}{81}\varphi_{3,7}, \\
\phi &=\tfrac{5}{32}a_{64} - \tfrac{1}{28}\varphi_{2,6} +\tfrac{36}{175}\varphi_{2,7}- \tfrac{48}{25}\varphi_{3,7} + \tfrac{6}{175}\varphi_{4,6} + \tfrac{192}{35}\varphi_{4,7}+6\varphi_{4,8}, \\
a_{85}&=\tfrac{208}{3}\varphi_{3,8}-\tfrac{16}{3}\varphi_{2,8}-40\phi, \qquad
a_{86}=-\tfrac{250}{3}\varphi_{3,8}+\tfrac{250}{21}\varphi_{2,8} + \tfrac{250}{7}\phi, \\
a_{87}&=-27\varphi_{3,8}+\tfrac{27}{14}\varphi_{2,8} + \tfrac{135}{7}\phi, \qquad
b_6=\tfrac{125}{14}\varphi_2-\tfrac{625}{14}\varphi_3+\tfrac{1125}{14}\varphi_4, \\
b_7& =-\tfrac{27}{14}\varphi_2+\tfrac{162}{7}\varphi_3-\tfrac{405}{7}\varphi_4, \qquad
b_8=\tfrac{1}{2}\varphi_2-\tfrac{13}{2}\varphi_3+\tfrac{45}{2}\varphi_4.
\end{aligned}
\end{equation*}
%%%-----------------------------------------
 \section{Numerical experiment} \label{sc6}
In this section, we illustrate the sharpness of our error bound in Theorem~\ref{th4.1} with a numerical example that was also used in \cite{HO05b}.
Consider the semilinear parabolic problem
\begin{equation} \label{example1}
\partial_t  u - \partial_{xx} u =\frac{1}{1+u^2}+\Phi (x,t)
\end{equation}
 for $u=u(x,t)$ on the unit interval $[0,1]$ and  $t \in [0,1]$, subject to homogeneous Dirichlet boundary conditions. The source function $\Phi$ is chosen in such a way that the exact solution of the problem is $u(x,t)=x(1-x)\ee^t$. This problem fits in our framework with $X=L^2([0,1])$. We discretize (\ref{example1}) in space by standard finite differences with 200 grid points. This yields a stiff system of the form \eqref{eq2.3a}. We integrate this problem in time by using the new integrator $\mathtt{expRK5s8}$. The errors are measured in a discrete $L^2$ norm at $t=1$. The order plot is displayed in Figure~\ref{fig2} in a double-logarithmic diagram. It is in perfect agreement with Theorem \ref{th4.1}.
\begin{figure}[ht!]
\begin{center}
\includegraphics[scale=0.5]{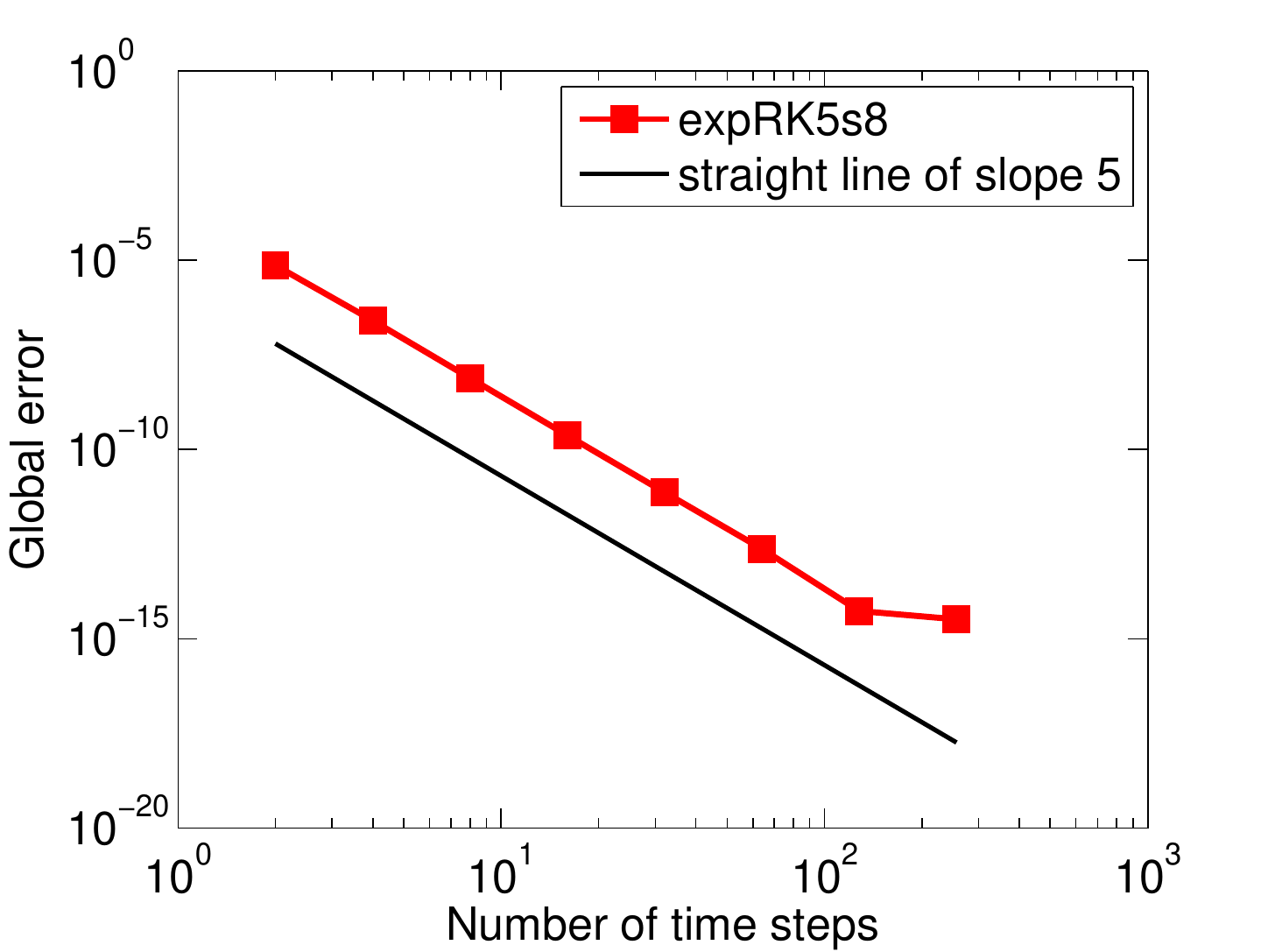}
\end{center}
\caption{\label{fig2} Order plot of the new exponential Runge--Kutta method of order~5, $\mathtt{expRK5s8}$, when applied to example~(\ref{example1}). The errors are plotted against the number of employed time steps. For comparison, we added a straight line of slope 5.}
\end{figure}
%%%%-------------------
%\\
%{\bf References}
%\section*{References}
\bibliographystyle{elsarticle-num}

\begin{thebibliography}{00}
\bibitem{Butcher64}
J.C.~Butcher, On Runge--Kutta processes of high order, J. Austral. Math. Soc. 4 (1964), 179--194.

\bibitem{F78}
A.~Friedli,
Verallgemeinerte Runge--Kutta Verfahren zur L\"osung steifer Differential\-glei\-chungssysteme,
in: R. Burlirsch, R. Grigorieff and J. Schr\"odinger (Eds.), Numerical Treatment of Differential Equations,
Springer, Berlin, 1978, pp.~35--50.

\bibitem{DP11}
G.~Dimarco, L.~Pareschi,
Exponential Runge--Kutta methods for stiff kinetic equations, SIAM J. Numer. Anal., 49 (2011), pp.~2057--2077.

\bibitem{D09}
G.~Dujardin, Exponential Runge--Kutta methods for the Schr\"odinger equation, Appl. Numer. Math., 59 (2009), pp.~1839--1857.

\bibitem{H81}
D.~Henry, Geometric Theory of Semilinear Parabolic Equations, Springer, Berlin, Heidelberg, 1981.

\bibitem{HLS98}
M. Hochbruck, C. Lubich, H. Selhofer,
Exponential integrators for large systems of differential equations,
SIAM J. Sci. Comput., 19 (1998), pp.~1552--1574.

\bibitem{HO10}
M. Hochbruck, A. Ostermann,
Exponential integrators, Acta Numerica, 19 (2010), pp.~209--286.

\bibitem{HO05a}
M. Hochbruck, A. Ostermann,  Exponential Runge--Kutta  methods for parabolic problems,  Appl. Numer. Math., 53 (2005), pp.~323--339.

\bibitem{HO05b}
M. Hochbruck, A. Ostermann,
Explicit exponential Runge--Kutta methods for semilinear parabolic problems,
SIAM J. Numer. Anal., 43 (2005), pp.~1069--1090.

\bibitem{KT05}
A.-K. Kassam, L.N. Trefethen,
Fourth-order time stepping for stiff PDEs,
SIAM J. Sci. Comput., 26 (2005), pp.~1214--1233.

%\bibitem{L12}
%V.T. Luan, High-order exponential integrators, Ph.D. thesis, University of Innsbruck (in preparation, 2012).

\bibitem{LO12b}
V.T. Luan, A. Ostermann,
Stiff order conditions for exponential Runge--Kutta methods of order five,
to appear in: Proceedings of the Fifth International Conference on High Performance Scientific Computing, 2012, Hanoi, Vietnam.

\bibitem{PAZY83}
A. Pazy,
Semigroups of Linear Operators and Applications to Partial Differential Equations,
Springer, New York, 1983.

\end{thebibliography}

\end{document}